\documentclass[12pt,a4paper,leqno,twoside]{amsart}

\usepackage{latexsym}
\usepackage{amsmath}
\usepackage{amsfonts}
\usepackage{amssymb}
\input xy
\xyoption{all}

\numberwithin{equation}{section}

\theoremstyle{plain}
\newtheorem{thm}{Theorem}[section]
\newtheorem{cor}[thm]{Corollary}

\newtheorem{prop}[thm]{Proposition}

\theoremstyle{definition}

\newtheorem{remark}[thm]{Remark}

\def\Sp{{\rm Sp}}  
\def\SU{{\rm SU}}  
  
\def\SO{{\rm SO}}  
\def\Spin{{\rm Spin}}
\def\Spinc{{\rm Spin}^c}
\def\U{{\rm U}}
\def\O{{\rm O}}

\def\KO{K{\rm O}}

\def\Sq#1{{\rm Sq}^{#1}}
\def\Z{{\mathbb Z}}
\def\R{{\mathbb R}}
\def\C{{\mathbb C}}
\def\H{{\mathbb H}}
\def\Hl{{\mathbb H}_{\lambda}}

\def\bar#1{\overline{#1}}   
\def\mod#1{(\hbox{\rm mod}\, {#1})}
\def\oo{\mathfrak{o}}

\begin{document}

\title[Obstruction theory on $7$-manifolds]{Obstruction theory on $7$-manifolds}

\author[M.~\v Cadek, M.~C.~Crabb and T. Sala\v c]{MARTIN \v CADEK, MICHAEL
CRABB and TOM\' A\v S SALA\v C}

\address{\newline Department of Mathematics, Masaryk University,
Kotl\' a\v rsk\' a 2, 611 37 Brno, Czech Republic}
\email{cadek@math.muni.cz}

\address{\newline Institute of Mathematics, University
of Aberdeen, Aberdeen AB24 3UE, U.K.}
\email{m.crabb@abdn.ac.uk}

\address{\newline Charles University, Faculty of Mathematics and Physics, 
Sokolovsk\' a 83 186 75 Prague, Czech Republic}
\email{salac@karlin.mff.cuni.cz}

\date {September 19, 2019}




\thanks{The research of the third author was supported by the grant 17-01171S
of the Grant Agency of the Czech Republic.}

\abstract{This paper gives a uniform, self-contained,
and direct approach to a variety of obstruction-theoretic problems on
manifolds of dimension $7$ and $6$. We give necessary and sufficient cohomological 
criteria for the existence of various $G$-structures  on vector bundles over such manifolds especially using low dimensional representations of $\U(2)$.}
\endabstract

\maketitle

\section{Introduction}
Let $M$ be a connected, closed, smooth, spin$^c$ manifold of dimension
$m=6$ or $7$ and let $\xi$ be an $m$-dimensional oriented real vector bundle
over $M$ admitting a spin$^c$-structure. For various homomorphisms
$\rho:G\to \SO(m)$ from a compact Lie group $G$ to $\SO(m)$ we consider  the problem 
of reducing the structure group $\SO(m)$ of the vector bundle $\xi$ to $G$ via 
the representation $\rho$. Necessary and sufficient conditions for such 
reductions will be obtained in terms of the cohomology of $M$ and cohomology 
characteristic classes of $M$ and $\xi$. Thus as for methods and results
the present paper is a continuation of \cite{CCV1}. 

Most of our results depend on the existence of $2$-dimensional complex vector bundles
over low dimensional manifolds. So we can provide more or less complete answers for all homomorphisms  $\rho$ which are connected 
with low dimensional representations of the group $G=\U(2)$. 

Our results  complete the characterization of $m$-dimensional vector bundles over $m$-dimensional complexes ($m=6,7$) given in \cite{W} and the results on the existence of vector fields over $m$-dimensional manifolds in \cite{T}.  

We conclude this introduction by describing our results for 
$7$-manifolds in the case that $G$ is the group $\Sp (1)$ and the manifold $M$ and vector bundle $\xi$ are spin.

There are $4$ irreducible real $\Sp (1)$-modules of dimension
at most $7$: the Lie algebra $A_1$ of dimension $3$,
the defining $4$-dimensional module $E$ ($=\H$),
a module $A_2$ of dimension $5$,
and a module $A_3$ of dimension $7$.
We thus have, up to equivalence, the following seven $7$-dimensional
real $\Sp (1)$-modules and associated representations $\rho$:

\smallskip

\par\noindent
(i) $\R^7$, (ii) $E\oplus\R^3$, (iii) $A_1\oplus\R^4$, 
(iv) $E\oplus A_1$, 
(v) $A_1\oplus A_1\oplus\R$, (vi) $A_2\oplus\R^2$, (vii) $A_3$.

\begin{thm}\label{symplectic}
Let $\xi$ be a $7$-dimensional vector bundle with $w_1\xi =0$
and $w_2\xi =0$ over a $7$-dimensional spin manifold $M$. 
Then the structure group
of $\xi$ reduces from $\Spin (7)$ to $\Sp (1)$ through $\rho$ if
and only if the spin characteristic class 
$q_1(\xi )\in H^4(M;\,\Z )$ is divisible by

\smallskip

\par\noindent
{\rm (i)} $0$, {\rm (ii)} $1$, {\rm (iii)} $2$, {\rm (iv)} $3$,
{\rm (v)} $4$, {\rm (vi)} $10$, {\rm (vii)} $28$.
\end{thm}

\begin{cor}
A $7$-dimensional spin bundle $\xi$ over a spin 
$7$-manifold $M$ admits $4$ linearly independent sections if
and only if $w_4(\xi )=0$.
\end{cor}

\begin{proof}
This follows from case (iii), because $\rho_2(q_1(\xi ))=
w_4 (\xi )$.
\end{proof}

\begin{cor}
For any $7$-dimensional spin bundle $\xi$ over a spin
$7$-manifold $M$, the $8$-dimensional vector bundle 
$\R 1\oplus \xi$ admits the structure of a bundle of Cayley
algebras and, hence, the structure group of $\xi$ reduces
from $\SO (7)$ to $G_2$.
\end{cor}

See \cite[Section 4.4]{S} and the references cited there for the case 
of the tangent bundle.
\begin{proof}
According to case (ii), $\xi$ is isomorphic to $\R^3\oplus \mu$ for
some (left) $\H$-line bundle $\mu$. A Cayley multiplication can then be
written down on $\R 1\oplus\xi = \H \oplus \mu$ using the
quaternionic inner product on $\mu$ satisfying
$\langle au,bv\rangle = a\langle u,v\rangle \bar{b}$:
$$
(a,u)\cdot (b,v) = (ab-\langle v,u\rangle, \bar{a}v+bu),
$$
where $a,\, b\in \H$ and $u,\, v$ are vectors in a fibre of
$\mu$.
There is an associated principal $G_2$-bundle with fibre 
the bundle of algebra isomorphisms from the standard Cayley algebra
(with automorphism group $G_2$) to the fibre of $\H\oplus\mu$.
\end{proof}

\section{The spin characteristic classes}
In this section we recall some standard facts about spin and spin$^c$ vector bundles in low dimensions and their characteristic classes. The inclusion
\begin{equation}\label{iso1}
\SU(\infty) \to \Spin(\infty)
\end{equation}
induces an isomorphism $\pi_i(B\SU(\infty))\to \pi_i(B\Spin(\infty))$ for $i\le 5$. 
So for a spin vector bundle $\xi$ over a manifold $M$
we may define $q_1(\xi)\in H^4(M;\Z)$ to be the characteristic class corresponding in the above isomorphism to the negative of the second Chern class
 $-c_2$. In Section 2 of \cite{CCV1} it was shown that $q_1(\xi)$ is independent of the choice of the spin structure. 

The inclusion (\ref{iso1}) induces the inclusion
\begin{equation*}\label{iso2}
\U(\infty)\simeq \SU(\infty)\times_{\{\pm 1\}}\U(1)\to \Spin(\infty)\times_{\{\pm 1\}} \U(1)=\Spinc(\infty)
\end{equation*}
which yields again an isomorphism  $\pi_i(B\U(\infty))\to \pi_i(B\Spinc(\infty))$ for $i\le 5$.
A vector bundle $\xi$ over $M$ with a fixed $\Spinc(\infty)$-structure has one characteristic class $l\in H^2(M;\ \Z)$ which corresponds to the first Chern class $c_1$ and a second characteristic class $q_1\in H^4(M;\ \Z)$ corresponding to $-c_2$. However, in this case
both classes depends on the choice of the $\Spinc(\infty)$-structures. 

Let $\zeta$ and $\zeta'$ be  two
$\Spinc(\infty)$-structures of the vector bundle $\xi$. Over $M^{(5)}$ they correspond to complex vector bundles. Their difference $\zeta'-\zeta$ considered as a map to
$M\to B\Spinc(\infty)$  lifts to the fiber $B\U(1)$ of the fibration $B\Spinc(\infty)\to
B\SO(\infty)$. The inclusion $i: B\U(1)\to B\Spinc(\infty)$ induces the multiplication by two 
$$H^2(B\Spinc(\infty);\ \Z)\simeq\Z\overset{\times 2}\longrightarrow H^2(B\U(1);\ \Z) \simeq\Z.$$ 
Hence for fixed $\zeta$ and any $m\in H^2(M;\ \Z)$ we can choose $\zeta'$ in such a way that
$c_1(\zeta'-\zeta)=2m$. Since the choice of $\Spinc(\infty)$-structure $\zeta$ is determined by the choice of $c_1(\zeta)=l$, we can now define
$$q_1(\xi;\ l)=-c_2(\zeta).$$
These classes for  different lifts are related by the formula
$$q_1(\xi;\ l+2m)=q_1(\xi;\ l)-2lm-2m^2.$$
Moreover, $2q_1(\xi;l)=p_1(\xi)-l^2$ and $\rho_2(q_1(\xi;l))=w_4(\xi)$.

\section{Spin$^c$ structures on 7-manifolds}
For spin${}^c$ manifolds of dimension $7$ we shall use the following 
general result on manifolds of dimension 
$m \equiv 3\, ({\rm mod}\, 4)$.

\begin{thm}\label{cs} 
Let $M$ be an orientable closed manifold
of dimension $4k+3$.
\begin{enumerate}
\item [(i)] 
Suppose that $\xi$ is a $(4k+3)$-dimensional oriented real vector bundle over $M$ such that $w_2(\xi )=w_2(M)$. Then $\xi$
splits as a direct sum $\xi'\oplus\R$ with 
$e(\xi')=0\in H^{4k+2}(M;\, \Z)$. 
\item [(ii)]
Suppose that $\xi'$ is a $(4k+2)$-dimensional oriented vector bundle 
over $M$ such that
$w_2(\xi' )=w_2(M)$ and $e(\xi' )=0\in H^{4k+2}(M;\, \Z )$.
Then $\xi'$ has a nowhere-zero section.
\item [(iii)]
Suppose that $\xi''$ is a $(4k+1)$-dimensional oriented vector bundle 
over $M$ such that
$w_2(\xi'' )=w_2(M)$ and $e(\xi'')=0\in H^{4k+1}(M;\, \Z )$. 
Then $\xi''$ has a nowhere-zero section.
\end{enumerate}
\end{thm}

Most of this may be found in \cite{CS}; see also \cite{T}.
For the sake of completeness we include a proof using the
$K$-theoretic methods introduced by Atiyah and Dupont \cite{AD}
in Section 6.

Using the  properties of the spin$^c$ characteristic class $q_1$ we apply   
Theorem \ref{cs} to a connected closed spin$^c$ manifold $M$ 
of dimension 7. If $\xi$ is an oriented $7$-dimensional vector bundle 
over $M$ with $w_2(\xi)=w_2(M)=\rho_2 l$, then 
$$\delta^*w_4(\xi)=\delta^*\rho_2q_1(\xi;l)=0,$$
where $\delta^*:H^4(M;\, \Z/2)\to H^5(M;\, \Z)$ is the Bockstein homomorphism.
By parts (i) and (ii), $\xi$ splits as $\xi''\oplus \R^2$, and then we can apply part (iii) to $\xi''$, because $e(\xi'')=\delta^*w_4(\xi'')=0$.
So $\xi$ has three linearly independent
cross sections, that is the structure group of $\xi$ reduces to
$\SO(4)<\SO(7)$.
In fact, under the same assumptions more is true.

\begin{prop}\label{eta}
Let $\xi$ be a $7$-dimensional vector bundle over a closed, connected,
smooth, spin$^c$ $7$-manifold $M$.
If $w_2(\xi)=w_2(M)=\rho_2 l$ for an  $l\in H^2(M;\Z)$, then there is
a $2$-dimensional complex bundle $\eta$ over $M$ with $c_1(\eta)=l$ such that
$\xi\cong\eta\oplus \R^3$. Moreover, $c_2(\eta)=-q_1(\xi;\ l)$.
\end{prop}

\begin{proof}
The inclusion
$$\SU(2)\cong\Sp(1)=\Sp(1)\times 1\hookrightarrow \Sp(1)\times\Sp(1)\cong
\Spin(4)\hookrightarrow \Spin(5)$$
induces an isomorphism on homotopy groups $\pi_i$ for $i\le 5$ and an
epimorphism on $\pi_6$.   Indeed, it can be identified with the inclusion $\Sp(1) \hookrightarrow \Sp(2)$, $g\mapsto \begin{pmatrix}g&0\\0&1\end{pmatrix}$, with quotient $S^7$. The same applies to the induced inclusion
$$\U(2)\cong \SU(2)\times_{\{\pm 1\}} \U(1)\hookrightarrow
\Spin(5)\times_{\{\pm 1\}}\U(1)=\Spinc(5).$$
By (i) and (ii) of Theorem \ref{cs} the structure
group $\Spinc(7)$ of $\xi$ can be reduced to $\Spinc(5)$. Since the
inclusion $\U(2)\hookrightarrow \Spinc(5)$ is a $6$-equivalence there is a
$2$-dimensional complex bundle $\eta$ such that 
$\eta\oplus\R^3\cong\xi$. 
In the previous section it was shown that the $\Spinc(5)$-structure of $\xi$ can 
be chosen in such a way that $c_1(\eta)=l$ since $w_2(\xi)=\rho_2l$. Then according to the definition of the spin$^c$ characteristic class $c_2(\eta)=-q_1(\xi;\ l)$.
\end{proof}

Notice that the above inclusion $\U(2)\hookrightarrow
\Spinc(5)\hookrightarrow \Spinc(7)$ is a lift of the standard inclusion
$$\U(2)\hookrightarrow\SO(4)\hookrightarrow\SO(7).$$ 

\begin{remark}
Let us recall from \cite{CCV1} the notion of $\Hl$-bundle for a complex line 
bundle $\lambda$. This is a complex vector bundle which is a left module over 
the bundle $\Hl=\C\oplus\lambda$ of quaternion algebras. The structure groups
for $\Hl$-bundles are $\Sp(n)\times_{\{\pm 1\}}U(1)$. Since $\U(2)\cong
\SU(2)\times_{\{\pm 1\}}\U(1)\cong \Sp(1)\times_{\{\pm 1\}}\U(1)$ every 
$2$-dimensional complex bundle $\eta$ is naturally an $\Hl$-line bundle, where $\lambda$ is the determinant bundle $\Lambda^2(\eta)$ (so that 
$c_1(\lambda)=c_1(\eta)$). 
Hence Proposition \ref{eta} can be read in the following way:

\smallskip

{\it Let $\lambda$ be a complex line bundle over $M$ with $c_1(\lambda)=l$. Suppose that $w_1(\xi)=0=w_1(M)$ and $w_2(\xi)=\rho_2(l)=w_2(M)$. Then there is an $\Hl$-line bundle
$\eta$  with the Euler class $e(\eta)=-q_1(\xi;\ l)$ such that $\xi$ is isomorphic to $\eta\oplus\R^3$.}
\end{remark}

The crucial role for our obstruction theory on $7$-manifolds is played
by Propositions \ref{spinc7} and \ref{u2}. The first completes the characterization of $7$-dimensional vector bundles by characteristic classes in \cite{W}. It has been already
used in \cite{S} to obtain results on the existence of multisymplectic $3$-forms on $7$-dimensional manifolds.

\begin{prop}\label{spinc7}
Suppose that $M$ is a $7$-dimensional connected closed manifold (not necessarily 
spin$^c$). Consider two orientable real $7$-dimensional vector bundles $\xi$ and
$\xi'$ with $w_2(\xi)=w_2(\xi')=\rho_2(l)$ for some $l\in H^2(M;\Z)$.
Then $\xi$ and $\xi'$ are isomorphic if and only if
$q_1(\xi;l)=q_1(\xi';l)$. 
\end{prop} 

\begin{proof} Suppose that $q_1(\xi;l)=q_1(\xi';l)$. Then the proof of Proposition 2.5 in \cite{CCV1} shows that
the restrictions of $\xi$ and $\xi'$  to the $6$-skeleton $M^6$  of the manifold $M$ are stably isomorphic. Since the inclusion $B\SO(7)\hookrightarrow B\SO(\infty)$ induces isomorphisms $\pi_i(B\SO(7))\to \pi_i(B\SO(\infty))$ for $i\le 7$, the vector bundles $\xi$ and $\xi'$ are also unstably isomorphic over $M^6$. The vanishing of $\pi_6(\SO(7))$ says that there is no obstruction to extending an isomorphism over the whole manifold $M$.
\end{proof}

\begin{prop}\label{u2} Suppose that  $M$ is a connected closed smooth spin$^c$
$7$-manifold with $w_2(M)=\rho_2(l)$ for some $l\in H^2(M;\Z)$. Let $u\in
H^4(M;\Z)$. Then there is a $2$-dimensional complex bundle $\eta$ such that
$$c_1(\eta)=l,\quad c_2(\eta)=u.$$ 
If $\lambda$ is a complex line bundle over $M$ with $c_1(\lambda)=l$, the complex bundle $\eta$ admits the structure of an $\Hl$-line bundle with Euler class $e(\eta)=u$.
\end{prop}

\begin{proof} 
We construct a 7-dimensional vector bundle $\xi$ with $w_1(\xi)=0$,
$w_2(\xi)=\rho_2(l)$ and $q_1(\xi;l)=-u$. Then we can apply Proposition 
\ref{eta} to obtain $\eta$ with the prescribed properties.

If we can construct $\xi$ over the 5-skeleton $M^{(5)}$ we are done, because
$\pi_5(BO(\infty))=0$ and $\pi_6(BO(\infty))=0$, so that any stable bundle
over $M^{(5)}$ can be extended over $M$. And, of course, any stable bundle
over a 7-manifold can be reduced to dimension 7.

On $M^{(5)}$ we can take $\mu\oplus\lambda\oplus\R$ where $\lambda$ is a
complex line bundle with $c_1(\lambda)=l$ and $\mu$ is an $\H$-line
bundle such that $e(\mu)=u$. The existence of $\mu$ over $M^{(5)}$ follows
from the fact that the Euler class $e:B\SU(2)\to K(\Z,4)$ induces isomorphism on $\pi_i$ for 
$i\le 4$ and an epimorphism for $i=5$. (Notice that $\mu$ on $M^{(5)}$ is a
$2$-dimensional complex bundle with $c_1(\mu)=0$ and $c_2(\mu)=e(\mu)=u$.) 
Then $q_1(\xi;\ l)=-c_2(\mu\oplus\lambda)=-c_2(\mu)=-u$.
\end{proof}

We add also a characterization of three-dimensional complex vector bundles over manifolds of dimension less or equal to $7$. To keep the notation from \cite{CCV1},
we use the symbol $[X,Y]$ for pointed homotopy classes of maps from $X$ to $Y$ (although in our case the sets of pointed and unpointed homotopy classes are the same). In Section 2 of \cite{CCV1}, it has been shown:

\begin{prop}\label{u3} Let $M$ be a manifold of dimension $\le 7$. The image of the mapping determined by the first three Chern classes
$$(c_1,c_2,c_3): [M_+;\ B\U(3)]\to H^2(M;\ \Z)\oplus H^4(M;\ \Z)\oplus H^6(M;\ \Z)$$
is the set $\{(l,u,v)\ |\ Sq^2\rho_2(u)+\rho_2(lu)=\rho_2(v)\}$.
\end{prop}

This Proposition can be also used to give an alternative proof of Proposition \ref{u2}.

\begin{proof} [Another proof of Proposition \ref{u2}] To prove Proposition \ref{u2} from Proposition \ref{u3} we show first that $l\in H^2(M;\ \Z)$, $u\in H^4(M;\ \Z)$ and $0\in H^6(M;\ \Z)$ are the Chern classes of a three-dimensional complex vector bundle, $\beta$ say, over $M$. For every $x\in H^1(M;\ \Z/2)$ we have
$$xSq^2\rho_2u=Sq^2(x\rho_2(u))=w_2(M)(x\rho_2(u))=x\rho_2(lu),$$
which gives that $Sq^2\rho_2u=\rho_2(lu)$. By part (ii) of Theorem \ref{cs}, $\beta$ has a nowhere-zero section and so splits as a direct sum $\beta=\C\oplus\eta$, where $\eta$ is a two-dimensional complex vector bundle.
\end{proof}

\section{Representations of $\U(2)$ in dimension $7$}

Let $E$ be the standard $2$-dimensional complex representation of $\U(2)$ and write
$L=\Lambda^2E$. Then the irreducible complex representations  of $\U(2)$ can be listed as
$$V_{j,k}=S^jE\otimes_{\C} L^k, \quad j\ge 0,\ k\in\Z,$$
where $S^j$ is the $j$th symmetric power. The representation is real if and only if $j=2i$
is even and $k=-i$. In that case we write
$$V_{2i,-i}=\C\otimes A_i,\quad i\ge 0. $$
Notice that $\dim A_i=2i+1$ and that the centre $\mathbb T$ of $\U(2)$ acts trivially on $A_i$, so that $A_i$ is a module over $\U(2)/\mathbb T=\SO(3)$. The 5-dimensional representation $A_2$ has attracted the special interest of geometers, see \cite{ABF} or \cite{CM}.

Denote by $E\U(2)$ the universal principal $\U(2)$-bundle over $B\U(2)$. For a real or complex representation $V$ of the group $\U(2)$ we define characteristic classes of $V$ as  the characteristic classes of the associated vector bundle $E\U(2)\times_{\U(2)} V$. 
It is straightforward to calculate characteristic classes of a representation $V$ in terms of
$c_1(E)=l$ and $c_2(E)=e(E)=u$. If we restrict to the classifying space of a maximal torus, $E$ splits as the direct sum $L_1\oplus L_2$ of two $1$-dimensional representations and $L=L_1\otimes_{\C}L_2$. 

For any real representation $V$ of $\U(2)$ write
$$p_1(V)=-2a(V)u+b(V)l^2,$$
where $a(V)$, $b(V)\in \Z$. Thus $a$ and $b$ are linear in $V$: $a(V\oplus W)=a(V)+a(W)$, $b(V\oplus W)=b(V)+b(W)$ for two representations $V$ and $W$.
Since $\rho_2(p_1(V))=w_2^2(V)$, we get that 
$w_2(V)=b(V)\rho_2(l)$.
If $b(V)$ is even, then 
\begin{equation}\label{even}
q_1(V)=-a(V)u+\frac{1}{2}b(V)l^2;
\end{equation}
 if $b(V)$ is odd, then
\begin{equation}\label{odd}
q_1(V;l)=-a(V)u+\frac{1}{2}(b(V)-1)l^2.
\end{equation}

Using the fact that
$S^j(E)=\bigoplus_{0\le r\le j}L_1^{\otimes(j-r)}\otimes_{\C}L_2^{\otimes r}$, 
we get 
$$V_{2i,-i}=\C\otimes A_i=\bigoplus_{-i\le r\le i}(L_1\otimes_{\C} L_2^*)^{\otimes r}$$
which yields
$A_i=\R\oplus \bigoplus_{1\le r\le i}(L_1\otimes_{\C}L_2^*)^{\otimes r}$.  Since 
$p_1(A_i)=-c_2(V_{2i,-i})$,  a standard computation expresses the second Chern class
as a symmetric polynomial in $x_1=c_1(L_1)$ and $x_2=c_1(L_2)$. In particular,
\begin{equation*}\label{a}
b(A_i)=1^2+2^2+\dots +i^2=\frac{1}{6}i(i+1)(2i+1), \quad a(A_i)=2b(A_i)
\end{equation*}
and $w_2(A_i)$ equals to $0$ for $i\equiv 0,\ 3$ $\mod 4$, and to $\rho_2 l$ for 
$i\equiv 1,\ 2$ $\mod 4$.  

Now consider $V_{j,k}$ as a real representation. Then the computation starting with
\begin{align*}
p_1(V_{j,k})&=p_1\left (\bigoplus_{0\le r\le j}L_1^{\otimes(j-r+k)}\otimes_{\C}L_2^{\otimes(r+k)}\right)=\sum_{r=0}^j p_1\left(L_1^{\otimes(j-r+k)}\otimes_{\C}L_2^{\otimes(r+k)} \right)\\
&=\sum_{r=0}^j c_1^2\left(L_1^{\otimes(j-r+k)}\otimes_{\C}L_2^{\otimes(r+k)} \right)
=\sum_{r=0}^j\left((j-r+k)x_1+(r+k)x_2\right)^2
\end{align*}
leads to
\begin{equation*}\label{v}
b(V_{j,k})=\sum_{r=0}^j(r+k)^2 \quad\text{and}\quad a(V_{j,k})=
\frac{1}{2}\sum_{r=0}^j(j-2r)^2=\frac{1}{6}j(j+1)(2j+1).
\end{equation*}

The $7$-dimensional real representations $V$ of $\U(2)$ can be listed, for $r$, $s$, $t\in \Z$ (with some redundancy because $L^{\otimes r}$ and $L^{\otimes (-r)}$ are isomorphic 
over $\R$) as:

\bigskip
\renewcommand{\arraystretch}{1.3}
\centerline{
\begin{tabular}{|c|c|c|}
\hline
Representation $V$& $a(V)$ & $b(V)$  \\
\hline
$L^{\otimes r}\oplus L^{\otimes s}\oplus L^{\otimes t}\oplus \R$&$0$ &$ r^2+s^2+t^2$\\
\hline
$(L^{\otimes s}\otimes_{\C}E)\oplus L^{\otimes t}\oplus \R$ & $1$ & $s^2+(s+1)^2+t^2$\\
\hline
$A_1\oplus L^{\otimes s}\oplus L^{\otimes t}$ & $2$ & $1+s^2+t^2$ \\
\hline
$A_1\oplus(L^{\otimes s}\otimes_{\C}E)$ & $3$ & $1+s^2+(s+1)^2$\\
\hline
$(A_1\otimes_{\R}L^{\otimes s})\oplus\R$ & $4$ & $2+3s^2$ \\
\hline
$A_2\oplus L^{\otimes s}$ & $10$ & $5+s^2$\\
\hline
$A_3$ & $28$ & $14$\\
\hline
\end{tabular}}
\bigskip

Notice that  $A_1\otimes_{\R}L^{\otimes s}\simeq V_{2,s-1}$ as real vector bundles.
The general statement is

\begin{thm}\label{7u2}
Let $\xi$ be a $7$-dimensional spin$^c$ vector bundle over a closed, connected,
smooth, spin$^c$ $7$-manifold $M$ with $w_2(M)=w_2(\xi)$.
Let $V$ be a $7$-dimensional real representation of $\U(2)$
and let $l$ be a class in $H^2(M;\,\Z )$.
\begin{enumerate}
\item If $b(V)$ is odd, then $\xi$ is isomorphic to $P\times_{\U(2)}V$ for some principal $\U(2)$-bundle $P$ over $M$ with $c_1=l$ if and only if $\rho_2(l)=w_2(\xi )\,
(=w_2(M))$ and
$$q_1(\xi;\ l)-\frac{1}{2}(b(V)-1)l^2\in a(V)H^4(M;\ \Z).$$
\item If $b(V)$ is even and $\rho_2(l)=w_2(\xi )=w_2(M)$, then $\xi$ is isomorphic to $P\times_{\U(2)}V$ for some principal $\U(2)$-bundle $P$ over $M$ with $c_1=l$ if and only if $\rho_2(l)=0$ (so that
$\xi$ and $M$ are spin) and
$$q_1(\xi)-\frac{1}{2}b(V)l^2\in a(V)H^4(M;\ \Z).$$
\end{enumerate}
\end{thm}

\begin{proof} The conditions in both cases are necessary since if $c_2$ of a principal $\U(2)$-bundle $P$
is equal to an element $u\in H^4(M;\ \Z)$, the spin$^c$ characteristic class satisfies (\ref{even}) or (\ref{odd}).

Conversely, suppose that there is $u\in H^4(M;\ \Z)$ such that
$$q_1(\xi;\ l)=\frac{1}{2}(b(V)-1)l^2+ a(V)u, \quad\text{if }b(V)\text{ is odd},$$
or 
$$q_1(\xi)=\frac{1}{2}b(V)l^2+ a(V)u,\quad\text{if }b(V)\text{ is even and }\rho_2(l)=0.$$
According to Proposition \ref{u2} there is a $2$-dimensional complex bundle $\eta$
with  $c_1(\eta)=l$ and $c_2(\eta)=u$. Since $\eta=P\times_{\U(2)}E$ for some principal
$\U(2)$-bundle $P$, we get that the real vector bundle $\xi'=P\times_{\U(2)}V$ has $w_2(\xi')=\rho_2(l)=w_2(\xi)$ and $q_1(\xi';\ l)=q_1(\xi;\ l)$. Hence Proposition \ref{spinc7} implies that $\xi=\xi'$.

Notice that the assuption $\rho_2(l)=w_2(\xi)$ in (2) is needed to apply Proposition \ref{u2}.
\end{proof}

In special cases we obtain results on $\SO(3)$-structures. The real representations
of $\U(2)$ trivial on the centre $\mathbb T$ of the group $\U(2)$ determine representations of $\SO(3)=\U(2)/\mathbb T$.  The $7$-dimensional representations of $\SO(3)$ are:

\bigskip
\renewcommand{\arraystretch}{1.3}
\centerline{
\begin{tabular}{|c|c|c|}
\hline
Representation $V$ & $a(V)$ & $b(V)$\\
\hline
 $\R^7$ & $0$ & $0$\\
\hline
$A_1\oplus\R^4$ & $2$ & $1$\\
\hline
 $A_1\oplus A_1 \oplus\R$ & $4$ & $2$\\
\hline
$A_2\oplus\R^2$ & $10$ & $5$\\
\hline
 $A_3$ & $28$ & $14$\\
 \hline
\end{tabular}}
\bigskip

In view of the isomorphism between $\U (2)$ and $\Spin^c (3)$,
given $l$ and $V$, Theorem 4.1 gives necessary and sufficient
conditions for the existence of a $\Spin^c (3)$ vector bundle
$\alpha$ with $\Spin^c$ characteristic class equal to $l$ 
such that $\xi$ is isomorphic to the bundle associated with 
$\alpha$ and $l$ by the representation $V$.

We have to distinguish two cases.

\smallskip
\par\noindent
(i) If $b(V)$ is even, then the manifold $M$ and the bundle $\xi$
must be spin ($w_2(M)=w_2(\xi)=0$). Theorem 4.1 gives necessary and sufficient conditions for the
existence of a spin$^c$ bundle $\alpha$ with $l=2m$:
there exists $m\in H^2(M;\, \Z )$ such that 
$$q_1(\xi )-a(V)m^2=q_1(\xi )-2b(V)m^2=q_1(\xi )-\frac{1}{2}b(V)(2m)^2 \in 
a(V)H^4(M;\, \Z ).$$  

\smallskip
\par\noindent
(ii) If $b(V)$ is odd, then $w_2\alpha =w_2\xi (=w_2(M))$ is 
a necessary condition for the existence of $\alpha$.  For a representation
$V$ factoring through $\U(2)/\mathbb T=\SO (3)$, this condition on $w_2\alpha$
is necessary even without assuming that $\alpha$ is spin$^c$, because
the homomorphism $H^2(B\SO (3);\, \Z /2) \to H^2(B\U (2);\, \Z /2)$ is
injective.  So in this case Theorem 4.1 gives necessary and
sufficient conditions for the existence of an $\SO (3)$-bundle 
$\alpha$ such that $\xi$ is isomorphic to the bundle associated
with $\alpha$ by the representation $\SO (3) \to \O (7)$ giving $V$:
there exists $l\in H^2(M;\, \Z)$ such that $w_2(M)=\rho_2(l)$ and
$q_1(\xi ; l) -\frac{1}{2}(b(V)-1)l^2 \in a(V)H^4(M;\, \Z )$.
And, if such a bundle $\alpha$ exists it must satisfy $w_2\alpha =
w_2\xi$ and so be spin$^c$.

\begin{cor}\label{7so3}
Let $\xi$ be a $7$-dimensional spin$^c$ vector bundle over a closed, connected,
smooth, spin$^c$ $7$-manifold $M$ with $w_2(M)=w_2(\xi)$.
Then
\begin{enumerate}
\item $\xi$ is a trivial vector bundle if and only if $w_2(\xi)=0$ and $q_1(\xi)=0$;
\item $\xi$ always has $3$ linearly independent sections; 
\item $\xi$ has $4$ linearly independent sections if and only if $w_4(\xi)=0$;
\item there is a $3$-dimensional spin vector bundle $\alpha$ such that $\xi\cong\alpha\oplus\alpha\oplus\R$ if and only if $w_2(\xi)=0$ and $q_1(\xi)\in 4H^4(M;\ \Z)$;
\item the structure group of $\xi$ reduces to $\SO(3)$ through the homomorphism
$\SO(3)\to \SO(5)\subseteq \SO(7)$ corresponding to $A_2$ if and only if $w_4(\xi)=0$ and 
$p_1(\xi)\in 5H^4(M;\ \Z)$;
\item  the structure group of $\xi$ reduces to $\Spin(3)$ through the homomorphism
$$\Spin(3)\to\SO(3)\to \SO(7)$$ 
corresponding to $A_3$ if and only if $w_2(\xi)=0$ and $q_1(\xi)\in 28H^4(M;\ \Z)$.
\end{enumerate}
\end{cor}

Theorem \ref{symplectic} from the Introduction is an immediate consequence of this corollary.

\begin{remark} In \cite{ABF}, Theorem 3.2 states that for the existence of the irreducible $\SO(3)$-structure corresponding to the representation $A_2$ on tangent bundles of $5$-dimensional oriented manifolds the conditions $w_4(M)=0$ and $p_1(M)\in 5H^4(M;\, \Z)$ 
are necessary. Part (4) of Corollary \ref{7so3} shows that on $7$-dimensional spin$^c$ manifolds these conditions are also sufficient. 
\end{remark}

\begin{proof} The proofs of (1), (4) and (6) are covered by the considerations for $b(V)$ even  in (i) preceding Corollary \ref{7so3}.

\smallskip
\par\noindent
(2) follows  directly from Proposition \ref{eta}. 

\smallskip
\par\noindent
(3) If $\xi=\alpha\oplus\R^4$ for a $3$-dimensional vector bundle $\alpha$, then
$w_4(\xi)=0$ and $w_2(\alpha)=w_2(\xi)=\rho_2 l$. Hence $\alpha$ has the $\Spinc(3)$-structure given by the class $l\in H^2(M;\Z)$. The condition
$w_4(\xi)=0$ is equivalent to the condition 
$$q_1(\xi;l)=q_1(\xi;l)-\frac{1}{2}(1-1)l^2\in 2H^4(M;\mathbb Z)$$ 
which is the sufficient and necessary condition for the existence of the 
$\Spinc(3)$-structure corresponding to $A_1$ according to Theorem \ref{7u2}.

\smallskip
\par\noindent
(5) Let $\xi$ reduce to $\SO(3)$ via the representation $A_2$. Since $b(A_2)=5$,  according to (ii) preceeding Corollary \ref{7so3} the $3$-dimensional real  vector bundle $\alpha$ associated to the $\SO(3)$-structure has $w_2(\alpha)=w_2(\xi)=\rho_2(l)$. We can apply Theorem \ref{7u2} to get as a necessary and sufficient condition
$$q_1(\xi;l)-\frac{1}{2}(5-1)l^2\in 10 H^4(M;\Z).$$
This condition is equivalent to  $w_4(\xi)=\rho_2q_1(\xi;\ l)=0$ and $p_1(\xi)\in 5H^4(M;\mathbb Z)$ since
$$p_1(\xi)=2q_1(\xi;l)+l^2=2(q_1(\xi;l)-\frac{1}{2}(5-1)l^2)+5l^2.$$

\end{proof}

\begin{remark} Let $M$ be a manifold as above.  If we want to apply the previous results to the tangent of $M$ or the normal bundle of some immersion of $M$ into $\R^{14}$, we can use the following computation of $w_4$.
Let $v_i(M)$ denote the Wu classes of $M$. Since
$w_2(M)=\rho_2l$, $w_3(M)=\Sq{1} w_2(M)=0$. Now $w(M)=\Sq{}(v(M))$  and  $v_j(M)=0$ for $j\ge 4$. Further, $v_1(M)=w_1(M)=0$, $v_2(M)=w_2(M)$, $v_3(M)+\Sq{1}w_2(M)=w_3(M)=0$, so that $v_3(M)=0$, and $w_4(M)=\Sq{1}v_3(M)+\Sq{2}v_2(M)=v_2(M)^2=w_2(M)^2$. This implies that:
\begin{enumerate}
\item[(i)] For the tangent bundle $w_4(M)=0$ if and only if $w_2(M)^2=0$.
\item[(ii)] If $\nu$ is the normal bundle of an immersion, then $w_4(\nu)=0$.
\end{enumerate}
\end{remark}

We conclude this section with a result on the existence of $\U(3)$-structures.

\begin{prop}\label{7u3}
Suppose that $M$ is a connected, closed, smooth, spin$^c$ $7$-manifold and $\xi$ an oriented $7$-dimensional real vector bundle with $w_2(\xi)=w_2(M)=\rho_2l$. Then for any $u\in H^4(M;\ \Z)$ and $v\in 2H^6(M;\ \Z)$ there is a $3$-dimensional complex vector bundle $\zeta$ such that $c_1(\zeta)=l$, $c_2(\zeta)=u$, $c_3(\zeta)=v$. Moreover, $\xi$ is isomorphic to $\zeta\oplus\R$ if and only if $q_1(\xi;\ l)=-u$.
\end{prop}

\begin{proof}  According to Proposition \ref{u3} the condition
$$Sq^2\rho_2(u)+\rho_2(lu)= w_2(M)\rho_2(u)+w_2(M)\rho_2(u)=0=\rho_2(v)$$
is sufficient  for the existence of a $3$-dimensional
complex bundle $\zeta$ with the Chern classes $c_1(\zeta)=l$, $c_2(\zeta)=u$ and $c_3(\zeta)=v$. 

Since $w_2(\xi)=\rho_2(l)=w_2(\zeta)$, the real vector bundles $\xi$ and $\zeta\oplus \R$ are isomorphic if and only if $q_1(\xi;l)=q_1(\zeta\oplus\R;l)=-c_2(\zeta)=-u$ by
Proposition  \ref{spinc7}.
\end{proof}

\section{Dimension $6$}

Now suppose that $M$ is a connected closed manifold of dimension  $6$. 
Let $\xi$ and $\xi'$ be two $m$-dimensional vector bundles over $M$ with
$w_2(\xi)=w_2(\xi')=\rho_2(l)$ for some $l\in H^2(M;\Z)$ and with 
$q_1(\xi;l)=q_1(\xi';l)$. Consider the $7$-dimensional manifold 
$N=M\times S^1$ and the $7$-dimensional bundles $p^*(\xi)\oplus \R$ 
and $p^*(\xi')\oplus \R$ where $p:N\to M$ is the obvious projection.
According to Proposition \ref{spinc7} these two bundles are isomorphic. 
Hence $\xi$ and $\xi'$ are stably isomorphic.  

For even dimensions the Euler class will distinguish 
stably isomorphic bundles. This fact has been considered to be well known for a long time; see the sentence following Lemma 2 in \cite{W}. Unfortunately, we have found only one reference for its proof, the relatively recent paper \cite{HRS}, (Theorem 3.9). This proof uses the Moore-Postnikov tower for the map $B\SO(m)\to B\SO(\infty)$ induced by inclusion. For the sake of completeness we include an alternative  proof  in  the Appendix. Using this fact  we get

\begin{prop}\label{spinc6}
Suppose that $M$ is a connected closed manifold of dimension $6$. 
Consider two orientable $6$-dimensional real vector bundles $\xi$ and
$\xi'$ with $w_2(\xi)=w_2(\xi')=\rho_2(l)$ for some $l\in H^2(M;\Z)$.
Then $\xi$ and $\xi'$ are isomorphic if and only if
$$q_1(\xi;l)=q_1(\xi';l),\quad\text{and}\quad e(\xi)=\pm e(\xi').$$
\end{prop}

Further we need the following analogue of Proposition \ref{u2}.

\begin{prop}\label{over6}
Suppose that  $M$ is a connected closed spin$^c$
$6$-manifold with $w_2(M)=\rho_2(l)$ for some $l\in H^2(M;\Z)$. Let $u\in
H^4(M;\Z)$ be arbitrary.  Then there is a $2$-dimensional complex bundle $\eta$ such that
$$c_1(\eta)=l,\quad c_2(\eta)=u.$$ 
If $\lambda$ is a complex line bundle with $c_1(\lambda)=l$, then $\eta$ is an $\Hl$-line bundle with $e(\eta)=u$.
\end{prop}

\begin{proof} Let $i:M\to M\times S^1$ be the inclusion of $M\times\{*\}$ and $p:M\times S^1\to M$ the projection. According to Proposition \ref{u2} there is a $2$-dimensional complex bundle
$\eta'$ over $M\times S^1$ with $c_1(\eta')=p^*(l)$ and $c_2(\eta')=p^*(u)$. Then $\eta=i^*(\eta')$ over $M$ has the prescribed properties.
\end{proof}

Now we can follow the same lines as in Section 4 to obtain results on the existence of 
$G$-structures on vector bundles over $6$-manifolds.

Consider a real $6$-dimensional representation $V$ of $\U(2)$.  As in the previous section
the first Pontrjagin class of the associated vector bundle over $B\U(2)$ 
is  determined by integers $a(V)$, $b(V)$ and the Euler class is given by  integers $c(V)$, $d(V)$ where 
$$e(V)=c(V)lu+d(V)l^3.$$ 
The list of these representations is as follows:

\bigskip
\renewcommand{\arraystretch}{1.3}
\centerline{
\begin{tabular}{|c|c|c|c|c|}
\hline
Representation $V$ & $a(V)$ & $b(V)$ & $c(V)$ & $d(V)$\\
\hline
 $L^{\otimes r}\oplus L^{\otimes s}\oplus L^{\otimes t}$ & $0$ & $r^2+s^2+t^2$ & $0$
 & $rst$\\
 \hline
 $(L^{\otimes s}\otimes_{\C}E)\oplus L^{\otimes t}$ & $1$ & $s^2+(s+1)^2+t^2$ & $t$ & $st(s+1)$\\
\hline
 $A_1\oplus L^{\otimes s}\oplus \R$ & $2$ & $1+s^2$ & $0$ & $0$\\
 \hline
$A_1\otimes_{\R}L^{\otimes s}$ & $4$ & $2+3s^2$ & $4s$ & $(s^2-1)s$\\
\hline
$A_2\oplus \R$ & $10$ & $5$ & $0$ & $0$\\
\hline
\end{tabular}}
\bigskip

\begin{thm}\label{6u2}
Let $\xi$ be a $6$-dimensional spin$^c$ vector bundle over a closed, connected,
smooth, spin$^c$ $6$-manifold $M$ with $w_2(M)=w_2(\xi)$.
Let $V$ be a $6$-dimensional real representation of $\U(2)$,
and let $l$ be a class in $H^2(M;\,\Z )$.
\begin{enumerate}
\item If $b(V)$ is odd, then  $\xi$ is isomorphic to $P\times_{\U(2)}V$ for some principal $\U(2)$-bundle $P$ over $M$ with $c_1=l$ if and only if $\rho_2(l)=w_2(\xi )\, (=w_2(M))$ and there is a class $u\in H^4(M;\ \Z)$ such that  
$$q_1(\xi;\ l)=\frac{1}{2}(b(V)-1)l^2- a(V)u\quad \text{and} \quad \pm e(\xi)=
c(V)lu+d(V)l^3.$$
\item  If $b(V)$ is even and $\rho_2(l)=w_2(\xi )=w_2(M)$, 
then  $\xi$ is isomorphic to $P\times_{\U(2)}V$ for some principal $\U(2)$-bundle $P$ over $M$ with $c_1=l$ if and only if $\rho_2(l)=0$ (so that $\xi$ and $M$ are spin) 
and there is a class $u\in H^4(M;\ \Z)$ such that  
$$q_1(\xi)=\frac{1}{2}b(V)l^2- a(V)u\quad \text{and} \quad \pm e(\xi)=c(V)lu+d(V)l^3.$$
\end{enumerate}
\end{thm}

In special cases we obtain

\begin{cor}\label{Gon6}
Suppose that $M$ is a closed, connected, smooth, spin$^c$ manifold of
dimension $6$. Let $\xi$ be a $6$-dimensional spin$^c$ real vector bundle over
$M$ with $w_2(\xi)=w_2(M)$.
\begin{enumerate}
\item There is always a $3$-dimensional complex vector bundle $\mu$ over $M$ such
that $\xi\cong\mu$ as real vector bundles. Moreover, $\mu$ can be chosen in
such a way that $c_1(\mu)=l$, where $\rho_2(l)=w_2(\xi)$,  $c_2(\mu)=-q_1(\xi;l)$ and $c_3(\mu)=e(\xi)$.
\item There is a $2$-dimensional complex vector bundle $\eta$ such that
$\xi\cong\eta\oplus\R^2$ as real vector bundles if and only if $e(\xi)=0$. Moreover, $\eta$
can be chosen in such a way that $c_1(\eta)=l$, where $\rho_2(l)=w_2(\xi)$, and $c_2(\eta)=-q_1(\xi;l)$.
\item There is a $3$-dimensional real oriented vector bundle $\alpha$ such that
$\xi\cong\alpha\oplus\R^3$ if and only if $e(\xi)=0$ and $w_4(\xi)=0$. 
\item  There is a $3$-dimensional real spin vector bundle 
$\alpha$ such that
$\xi\cong\alpha\oplus\alpha$ if and only if $w_2(\xi)=0$, $e(\xi)=0$ and $q_1(\xi)\in 4H^4(M;\ \Z)$.
\item There is a $3$-dimensional real vector bundle $\alpha$ such that
that $\xi\cong S^2\alpha$ (the structure group $\SO(6)$ of $\xi$ reduces to $\SO(3)$ via the representation $A_2\oplus\R$) if and only if 
$e(\xi)=0$, $w_4(\xi)=0$ and $p_1(\xi)\in 5H^4(M;\ \Z)$.
\item There is a complex line bundle $\lambda$ with $c_1(\lambda)=l$, where $\rho_2(l)=w_2(\xi)$, such that
$\xi\cong\lambda\oplus\R^4$ as real vector bundles if and only if
$e(\xi)=0$ and $q_1(\xi;l)=0$.
\item $\xi$ is a trivial bundle if and only if $w_2(\xi)=0$, $q_1(\xi)=0$ and $e(\xi)=0$.
\end{enumerate}
\end{cor}

\section{Obstruction theory on $(4k+3)$-dimensional manifolds}
In this section we prove Theorem \ref{cs} using the
$K$-theoretic methods introduced by Atiyah and Dupont.
A proof in the case that $M$ is simply connected is given in
\cite[Theorem 0.4]{CS}, and a proof for the case that $\xi$
(or $\xi'\oplus\R$ or $\xi''\oplus\R^2$) is the tangent bundle
of $M$ is contained in \cite{AD, dupont}. 

We use $\Z /2$-equivariant $K$-theory and write $L$ in this section
for the $\Z /2$-module $\R$ with the action $\pm 1$.
The reduced real
$K$-theory of the Thom space of a virtual vector bundle
$\alpha$ over a compact Hausdorff 
space $X$ will be written as $\KO^*(X;\,\alpha )$.

We may assume that $M$ is connected and write its dimension
as $m=4k+3$; the tangent bundle of $M$ is denoted by $\tau M$.
The proof hinges on the vanishing of the equivariant $K$-groups
$\KO^{-m}_{\Z /2}(*;\, -nL)$ for any $n\in\Z$, see table 3.1 in \cite{euler}. 
This paper is probably the best source for a discussion of the 
$\Z /2$-equivariant $\KO$-obstruction theory and the usage of the $\KO_{\Z/2}$-Euler class.

\begin{proof}[Proof of {\rm (i)}]
Since $e(\xi )=0$, the sphere bundle $S(\xi )$ has a section and
so $\xi$ splits as $\xi'\oplus\R$ for some oriented vector bundle
$\xi'$ of dimension $4k+2=m-1$. We show first that $w_{m-1}(\xi')=0$
and do this by proving that $(x\cdot w_{m-1}(\xi'))[M]=0$
for any class $x\in H^1(M;\, \Z /2)$.
Let $\nu$ be a smooth real line bundle over $M$ with $w_1(\nu )=x$.
Choose a generic smooth section of $\nu$ with zero-set a submanifold
$N\subseteq M$ with normal bundle $\nu\,  |\, N$.
We have to show that $w_{m-1}(\xi' \,|\, N)[N]=0$.
Now there is a section $s'$ of $\xi'\, |\, N$ with only finitely
many zeros 
(or, indeed, with at most one zero in each component of $N$).  
To each zero of $s'$ we can assign an index in
$\Z /2$ as a degree (mod $2$), and the sum of these local
indices is equal, by a cohomology Hopf theorem, to 
$w_{m-1}(\xi'\, |\, N)[N]$. To show that the sum is zero, we 
shall repeat the computation using $\KO$-theory.

Consider the commutative diagram
$$
\begin{matrix}
\KO^0_{\Z /2}(N;\, -L\otimes\xi') &\to&  
\KO^0(N;\, -\xi') \\
\noalign{\smallskip}
{\cdot\eta (\nu )}\downarrow\phantom{\cdot\eta (\nu )}&&
\phantom{\cdot\eta (\nu )}\downarrow{\cdot\eta (\nu )}\\
\noalign{\smallskip}
\KO^0_{\Z /2}(N;\, \nu -L\otimes\xi') &\to&  
\KO^0(N;\,\nu -\xi') \\
\noalign{\smallskip}
{\simeq}\downarrow\phantom{\simeq} &&\phantom{\simeq} \downarrow{\simeq} \\
\KO^{-m}_{\Z /2}(N;\, -(m-1)L -\tau N)
&\to&\KO^{-m}(N; -(m-1)\R -\tau N) \\
\noalign{\smallskip}
{\pi_!}\downarrow\phantom{\pi_!} && \phantom{\pi_!}\downarrow{\pi_!} \\
\noalign{\smallskip}
\KO^{-m}_{\Z /2} (*;\, -(m-1)L) =0
&\to&\Z /2=
\KO^{-m}(*;\, -(m-1)\R )\, ,
\end{matrix}
$$
in which 
the Hopf element $\eta (\nu ) \in \KO^0(N;\, \nu )$
is constructed from a universal class $\eta (L)$
generating $\KO^0_{\Z /2}(*;\, L)\simeq \Z$ and restricting to the
Hopf generator $\eta$ of $\KO^0(*;\, \R )=(\Z /2 )\eta$,
the vertical isomorphisms are given by Bott isomorphisms
determined by the orientation of $\xi$ and
a choice of spin structure for $\xi - \tau M$ and
so for $(\xi' - \nu)|N -\tau N$,
and $\pi_!$ is the Umkehr homomorphism for the projection
$\pi : N \to *$ to a point.
The horizontal maps restrict from the $\Z /2$-equivariant to the
non-equivariant theory.
The $K$-theory Euler class $\gamma (\xi'\, |\, N)\in \KO^0(N;\, -\xi')$
maps to the sum of the local indices in $\Z /2$, because the relative
Euler class gives an isomorphism
$$
\pi_{m-2}(S^{m-2})=\Z \to \KO^0(D^{m-1},S^{m-2};\, -\R^{m-1})=\Z
$$
and multiplication by $\eta$
$$
\KO^0(D^{m-1},S^{m-2};\, -\R^{m-1})=\Z \to
\KO^{-1}(D^{m-1},S^{m-2};\, -\R^{m-1})=\Z/2
$$
is reduction (mod $2$).
Since $\gamma (\xi'\,|\, N)$ lifts in the commutative diagram
to the equivariant Euler class
$\gamma (L\otimes \xi'\, |\, N)\in 
\KO^0_{\Z /2}(N;\, -L\otimes \xi')$, the sum of the obstructions in 
$\Z /2$ must be zero,
and we have proved that $w_{m-1}(\xi')=0$.

Now if $s_0$ and $s_1$ are two sections of $S(\xi )$ with
complementary $(m-1)$dimensional bundles $\xi_0'$ and $\xi_1'$,
there is an associated difference class 
$d(s_0,s_1)\in H^{m-1}(M;\, \Z)$ with the property,
because $m-1$ is even, that
$e(\xi'_1)-e(\xi'_0)=2d(s_0,s_1)$.
Moreover, because the Hurewicz homomorphism $\pi_j(S^{m-1})
\to \tilde H_j(S^{m-1};\, \Z )$ is an isomorphism for $j<m$
and an epimorphism for $j=m$,
for every section $s_0$ and class $y\in H^{m-1}(M;\,\Z )$
there is a section $s_1$ such that $d(s_0,s_1)=y$.

We conclude from the vanishing of $w_{m-1}(\xi )$ that, 
for each $y\in H^{m-1}(M;\, \Z )$, there is
a section $s$ of $S(\xi )$ such that $e(\xi') =2y$.
\end{proof}
\begin{remark}
When $k=0$, $\xi'$ is a complex line bundle with $c_1(\xi')=0$
and so is trivial.  It follows that the bundle $\xi$ is trivial.
For the special case $\xi =\tau M$ we have reproved the classical
result that an oriented $3$-manifold is parallelizable. 
In (ii) and (iii), the bundles $\xi'$ and $\xi''$ are clearly
trivial if $k=0$.
\end{remark}
\begin{proof}[Proof of {\rm (ii)}]
The argument that follows is a reformulation of the method used in
\cite[Theorem 0.4]{CS}.

Consider the diagram
$$
\begin{matrix}
\KO^0_{\Z /2}(M;\, -L\otimes\xi') &\to&  
\KO^0(M;\, -\xi') \\
\noalign{\smallskip}
{\simeq}\downarrow\phantom{\simeq} &&\phantom{\simeq} \downarrow{\simeq} \\
\noalign{\smallskip}
\KO^{-m}_{\Z /2}(M;\, -(m-1)L -\tau M)
&\to&\KO^{-m}(M; -(m-1)\R -\tau M) \\
{\pi_!}\downarrow\phantom{\pi_!} && \phantom{\pi_!}\downarrow{\pi_!} \\
\noalign{\smallskip}
\KO^{-m}_{\Z /2} (*;\, -(m-1)L) =0
&\to&\Z /2=
\KO^{-m}(*;\, -(m-1)\R )\, ,
\end{matrix}
$$
in which the vertical Bott isomorphisms are again determined by the
orientation of $\xi'$ and a choice of spin structure for 
$\xi' -\tau M$ and $\pi_!$ is the Umkehr for the map $\pi : M\to *$.

The vanishing of $e(\xi')$ guarantees that there is a section $s$ of
$S(\xi')$ over the complement of the interior of an embeddded
disc $D^m \subseteq M$. 
The obstruction to extending $s$ over $M$ lies
in $\pi_{m-1}(S^{m-2}) =(\Z /2)\eta$ if $k>0$. 
(If $k=0$, the obstruction group is zero.)
It is detected by the $K$-theory Euler class
$\gamma (\xi')\in \KO^0(M;\, -\xi')$, 
because the relative Euler class gives an isomorphism
$$
\pi_{m-1}(S^{m-2}) \to \KO^0(D^m, S^{m-1};\, -(m-1)\R)=\Z /2\, .
$$
Since $\gamma (\xi')$ lifts in the commutative diagram above to
the $\Z /2$-equivariant Euler class $\gamma (L\otimes\xi')$,
the obstruction vanishes and $s$ extends to a section of $S(\xi')$
over $M$.
\end{proof}

\begin{proof}[Proof of {\rm (iii)}]
We may assume that $k>0$.
Choose a triangulation of the manifold $M$. The vanishing of 
$e(\xi'')$ guarantees that there is a section $s$ of $S(\xi'')$
over the $(m-2)$-skeleton $M^{(m-2)}$. Since $\pi_{m-2}(S^{m-3})
=(\Z /2)\eta$, there is an associated obstruction class
$\oo (s) \in H^{m-1}(M;\, \Z /2)$ to extending $s$ over $M$:
if $\oo (s)=0$, there is a section of $S(\xi'')$ over $M^{(m-1)}$
(coinciding with $s$ on $M^{(m-3)}$).

Following the method used to establish (i), we shall
show that $\oo (s)=0$.
Given $x\in H^1(M;\,\Z /2)$ we choose, as in (i), a line bundle 
$\nu$ and submanifold $N$. If we fix a triangulation of $N$, 
the inclusion of $N$ in $M$ is homotopic to a cellular map 
$f: N\to M$ and $f^*s$
on $N^{(m-2)}$ extends to a section of $f^*\xi''$ with zeros at the
barycentres of the $(m-1)$-cells.  Then $(x\cdot \oo (s))[M]=
\oo (f^*s)[N]$.

The $\KO$-theory Hopf theorem shows that $\oo (f^*s)[N]=0$. This time
the relative Euler class gives an isomorphism
$$
\pi_{m-2}(S^{m-3})=\Z /2 \to \KO^0(D^{m-1},S^{m-2};\, -\R^{m-2})
=\Z /2
$$
and multiplication by $\eta$ is an isomorphism
$$
\KO^0(D^{m-1},S^{m-2};\, -\R^{m-2})=\Z /2
\to \KO^{-1}(D^{m-1},S^{m-2};\, -\R^{m-2})=\Z /2\, .
$$

The vanishing of $\oo (s)$ implies that $S(\xi'')$ admits a section
on the complement of the interior of an embedded disc 
$D^m\subseteq M$. 
One shows finally, by the method used in (ii) (and \cite{CS}), that 
the obstruction in $\pi_{m-1}(S^{m-3})=(\Z /2)\eta^2$
to extending this section over the disc vanishes, because
the local obstruction map
$$
\pi_{m-1}(S^{m-3}) \to \KO^0(D^m,S^{m-1};\, -\R^{m-2})=\Z /2
$$
is an isomorphism.
\end{proof}
\begin{remark}
Massey showed in \cite{massey} that $w_{4k}(M)$ is the reduction
of an integral class.
(We have $w_{4k}(M) = Sq^{2k}v_{2k}(M) + Sq^{2k-1}v_{2k+1}(M)
=v_{2k}(M)^2+ Sq^1(Sq^{2k-2}v_{2k+1}(M))$.  But
$v_{2k}(\tau M)^2 = v_{4k}(\C \otimes\tau M)$ is integral, as is any class in the image of $Sq^1$.)
 It follows that, if
$\xi''$ is stably equivalent to $\tau M$, then
$e(\xi'')=\delta^*w_{4k}(M)=0$. In particular,
as is noted in \cite{dupont}, the manifold $M$ admits $3$
linearly independent vector fields.
\end{remark}


\section{Appendix}
In this final section we give an alternative proof of the following statement which is considered to be well known:

\begin{thm}\label{euler}
 Let $M$ be a connected oriented closed manifold of even dimension $m$ and let $\xi$ and $\xi'$ be two oriented real vector bundles over $M$ of dimension $m$ which are stably isomorphic. Then they are isomorphic  if and only if 
$$e(\xi)=\pm e(\xi').$$ 
\end{thm}

For the proof we need some preparatory considerations. 
Let $\zeta$ be an oriented real vector bundle of dimension $n+1$ over
a finite complex $X$.  Given two sections $s_0$, $s_1$ of the sphere 
bundle $S(\zeta)$, the difference class $d(s_0,s_1)\in H^n(X;\,\Z)$
is defined as follows. Choose a homotopy $\tilde s_t$, 
$0\leq t\leq 1$, from $s_0$ to $s_1$ through sections of the disc
bundle $D(\zeta )$. Regarding $\tilde s$ as a map 
$([0,1],\{ 0,1\})\times X\to (D(\zeta ),S(\zeta ))$, we define
the difference class $d(s_0,s_1)$ to be $\tilde s^*(u)\in
H^n(X;\, \Z)\cong H^{n+1}(([0,1],\{ 0,1\})\times X;\, \Z)$, the pullback
of the Thom class 
$u\in H^{n+1}(D(\zeta ),S(\zeta );\, \Z)$.

Clearly, $d(s_0,s_1)=0$ if $s_0$ and $s_1$ are homotopic through 
sections of $S(\zeta )$.  By elementary obstruction theory, the  converse is true if 
$\dim X \leq n$.

Now, if $s_2$ is a third section, it is elementary to check that
$d(s_0,s_2)=d(s_0,s_1)+d(s_1,s_2)$. So $d(s_0,s_0)=0$ and
$d(s_1,s_0)=-d(s_0,s_1)$.

For a section $s$ of $S(\zeta )$, let $\zeta (s)$ be the 
complementary $n$-dimensional vector bundle; it is oriented
by the orientation of $\zeta$. It follows from the definition
of the Euler class of $\zeta (s)$ as the pullback of the
Thom class in $H^n(D(\zeta (s)),S(\zeta (s));\, \Z)$ by the
zero-section $(X,\emptyset )\to (D(\zeta (s)),S(\zeta (s))$
that $e(\zeta (s))=d(-s,s)\in H^n(X;\, \Z)$ (with the correct
choice of sign convention).

Now $d(-s_0,s_0)+d(s_0,s_1)+d(s_1,-s_1)+d(-s_1,-s_0)=d(-s_0,-s_0)=0$.
From the definitions, we have $d(s_1,-s_1)=-d(-s_1,s_1)$ and
$d(-s_0,-s_1)=(-1)^{n+1}d(s_0,s_1)$. Hence
$e(\zeta (s_0))+d(s_0,s_1) -e(\zeta (s_1)) +(-1)^n d(s_0,s_1)=0$,
that is,
\begin{equation*}
e(\zeta (s_1))-e(\zeta (s_0))=(1+(-1)^n)d(s_0,s_1).
\end{equation*}

\begin{proof}[Proof of Theorem \ref{euler}] Put $\zeta=\xi\oplus\R\cong\xi'\oplus\R$.
The bundle $\xi$  can be considered as the complementary bundle $\zeta(s_0)$ of a section $s_0$ of the sphere bundle $S(\zeta)$ and $\xi'$ can be considered up to orientation as the complementary bundle $\zeta(s_1)$ of another section $s_1$. Hence
$$e(\xi)=e(\zeta(s_0)) \quad \text{and} \quad e(\xi')=\pm e(\zeta(s_1)).$$
Using the formula derived above we get that $e(\xi')\pm e(\xi')=e(\zeta(s_0))-e(\zeta(s_1))=0$ if and only if $d(s_0,s_1)=0$. This means that $s_0$ and $s_1$ are homotopic through sections of $S(\zeta)$. Such a homotopy determines, up to homotopy, an isomorphism between $\zeta(s_0)$ and $\zeta(s_1)$ as oriented vector bundles. So the vanishing of $d(s_0,s_1)$ implies that $\xi$ and $\xi'$ are isomorphic as vector bundles.
\end{proof}


\end{document}